\documentclass[abstract=on, 12pt]{scrartcl}
\usepackage[utf8]{inputenc}
\usepackage[T1]{fontenc}
\usepackage{hyperref}
\usepackage{lmodern}
\usepackage{graphicx}
\usepackage[a4paper]{geometry}
\usepackage{amsmath}
\usepackage{amssymb}
\usepackage{amsfonts}
\usepackage{amsthm}
\usepackage{parskip}
\usepackage{natbib}
\usepackage{caption}
\usepackage{subcaption}
\newcommand{\R}{\mathbb{R}}
\bibpunct{(}{)}{;}{}{}{,}
\usepackage[nottoc, notlof, notlot]{tocbibind}
\newtheorem{lem}{Lemma}
\newtheorem{prop}{Proposition}
\theoremstyle{remark}
\newtheorem{rmk}{Remark}

\begin{document}
\thispagestyle{empty}
\setcounter{page}{1}
\setcounter{secnumdepth}{3}

\title{On the large time behaviour of the solution  of an SDE driven by a Poisson Point Process}
\author{Elma Nassar 
and Etienne Pardoux}
\maketitle

\begin{abstract}
    We study a stochastic differential equation driven by a Poisson
    point process, which models continuous changes in a population's
    environment, as well as the stochastic fixation of beneficial mutations 
    that might compensate for this change. The fixation probability of a 
    given mutation increases as the phenotypic lag $X_t$ between the population
    and the optimum grows larger, and successful mutations are 
    assumed to fix instantaneously (leading to an adaptive jump).
    Our main result is that the process is transient (i.e., continued
    adaptation is impossible) if the rate of environmental change $v$
    exceeds a parameter $m$, which can be interpreted as the rate of
    adaptation in case every beneficial mutation gets fixed with probability 
    $1$. If $v < m$, the process is positive recurrent, while in the limiting 
    case $m=v$, null recurrence or transience depends upon additional 
    technical conditions. We show how our results can be extended to the 
    case of a time varying rate of environmental change. 
    \end{abstract}

\section{Introduction}
We study the large time behaviour of the solution of a scalar stochastic differential equation of the type
\begin{equation}
\label{eq:eqdedepart}
X_t=X_0 - v(t) + \int_{[0,t]\times\R\times[0,1]}\alpha\varphi(X_{s^-},\alpha,\xi)M(ds,d\alpha,d\xi),
\end{equation}
where $M$ is a Poisson point process  on $\R_+\times\R\times[0,1]$ with mean measure
$ds\ \nu(d\alpha)\ d\xi$ and $\varphi(x,\alpha,\xi)={\bf1}_{\{\xi\le g(x,\alpha)\}}$. 
The goal of our work is to understand how a population can adapt to a deterioration of its fitness, due for instance to continuous change in the climatic conditions, thanks to mutations which improve its adaptation to the new environment. $ds\ \nu(d\alpha)$
represents the rate of appearance of new mutations, while $g(x,\alpha)$ is the probability that a mutation $\alpha$, which is proposed while the population's phenotypic lag is given by $x$, gets fixed. We assume that $g(x,\alpha)\rightarrow 1$ when $x\rightarrow \pm \infty$ provided that 
$x\alpha<0$.

We start with the simple case $v(t)=vt$, with $v>0$. 
With the notation $m=\int_0^\infty \alpha\nu(d\alpha)$, in other words
$m$ is the mean movement to the right per time unit produced by the positive mutations 
if all of them get fixed, our first result says that the Markov process $X_t$ is positive recurrent 
if $m>v$, transient if $m<v$, with a speed of escape to infinity equal to $v-m$. The most 
interesting case is the limit situation $m=v$. We show that, depending upon the speed at which 
$m(x)=\int_0^\infty \alpha g(x,\alpha)\nu(d\alpha)$ converges to $m$ as $x\to-\infty$, the process
 can be either null recurrent or else transient with zero speed.

We then generalize our results to the case where $v(t)$ is a more general (and even possibly random) function of time.

Note that \citet{Ker86}  has studied similar questions in discrete time. 
Similar resuts for a SDE driven by Brownian motion with coefficients that do not depend upon the time variable would be easy to obtain. Here we use stochastic calculus and several ad hoc Lyapounov functions. Note that
the It\^o formula for processes with jumps leads to less explicit computations than in the Brownian case. To circumvent this difficulty, for the treatment of the delicate case $m=v$, we establish a stochastic inequality for  $C^2$ functions whose 
second derivative is either increasing or decreasing, exploiting the fact that all jumps have the
 same sign, see Lemma \ref{ito-ineq} in subsection \ref{m=v} below.

The paper is organized as follows. We define our model in detail in section \ref{sec2}, refering to models already studied in the biological literature. We establish existence 
and uniqueness of a solution to our equation in section \ref{sec3} (the result is not immediate since we do not
 assume that the measure $\nu$ is finite). Section \ref{sec4} is devoted to the large time behaviour of $X_t$ when $v(t)=vt$, successively with $m<v$, $m>v$, and $m=v$. Finally section \ref{sec5} is devoted to the 
 large time behaviour of $X_t$ when $v(t)$ takes a more general form, but $\overline{v}=\lim\limits_{t\to\infty}t^{-1}\int_0^tv(s)ds$ exists.

\section{The Model}\label{sec2} 

Our starting point is the model by \citet{KopH09b} of a population of
constant size $N$ that is subject to Gaussian stabilizing selection, with a
moving optimum that increases linearly at rate $v$.
The population is assumed to be
monomorphic at all times (i.e., its state is completely characterized
by $x$). Mutations arise according to a Poisson point process with 
intensity $ds \nu(d\alpha)$. Note that, in the model considered in~\cite{KopH09b},
\begin{equation}\label{eq:density}
\nu(d\alpha)=\frac{\Theta}{2}p(\alpha)d\alpha,
\end{equation}
which translates as follows: 
mutations appear at rate $\Theta/2 = N\mu$ 
(where $\mu$ is the \emph{per-capita} mutation rate and $\Theta = 2N\mu$ is a standard
population-genetic parameter), and their phenotypic effects $\alpha$ are
drawn from a distribution with density $p(\alpha)$. Whereas in our model, we do not impose 
that $\nu$ has a density, nor that it is a finite measure.
We neglect the possibility of
fixation of deleterious mutations. Yet even beneficial mutations have
a significant probability of being lost due to the effects of genetic
drift. A mutation with effect $\alpha$ that arises in a
population with phenotypic lag $x$ has a probability of fixation $g(x,\alpha)$
that satisfies 
\begin{enumerate}
\item $0\leq g(x,\alpha)\leq {\bf 1}_{\{\alpha x<0\}}\times {\bf 1}_{\{|\alpha|\leq 2|x|\}}$,
\item For all $\alpha\in\R$,  $g(x,\alpha)\uparrow {\bf 1}_{\{\alpha x<0\}}$, as
 $|x|\rightarrow\infty$,
 \item For any compact set $K\subset\R$, there exists $c_K>0$ such that for all $x,y\in\R$
 \begin{equation}\label{eq:cond3}
 \int_K |\alpha|\times |g(x,\alpha)-g(y,\alpha)|\leq c_K |x-y|.
 \end{equation}
\end{enumerate}
One popular model among theoretical biologists for the fixation probability is
\begin{equation}
    \label{eq:pfix}
    g(x,\alpha) = 
    \begin{cases}
	1 - \exp(-2s(x,\alpha)) \quad\text{ if } s(x,\alpha) > 0, \\
	0 \text{ otherwise}
    \end{cases}
\end{equation}
where 
\begin{equation}
\label{eq:selcoef}
    s(x,\alpha) = - \sigma[|\alpha|(2|x|-|\alpha|)]^+ \times \mathbf 1_{\{x\alpha<0\}}
 \approx\frac{\mathcal W(x+\alpha)}{\mathcal W(x)} - 1
    \end{equation}
is the selection coefficient; $ \mathcal W(x) = e^{-\sigma x^2}$ being the 
fitness of the population when the phenotypic lag is $x$ and $\sigma$ 
determining the strength of selection.
The expression of $g$ defined in \eqref{eq:pfix} and \eqref{eq:selcoef} satisfies the three conditions listed above
(see Lemma~\ref{lemcond3} for the proof of~\eqref{eq:cond3}), and is a good approximation of the fixation 
probability derived under a diffusion approximation \citep{Mal52, Kim62}, 
which is valid when the population size $N$ is large enough. 
Note that \citet{KopH09b} used the even simpler approximation
$g(x,\alpha) \approx 2s(x,\alpha)$ (\citealp{Hal27}; for more exact approximations
for the fixation probability in
changing environments, see \citealp{UecH11, PeiK12}). Once a mutation
gets fixed, it is assumed to do so instantaneously, and the phenotypic
lag $x$ of the population is updated accordingly. 

In the particular case $v(t)=vt$, the evolution of the phenotypic lag of the population
can be described by the following particular case of equation~\eqref{eq:eqdedepart}:
\begin{equation} 
\label{eq:Xt}
X_{t}=X_{0}-vt+\int_{[0,t]\times\mathbb{R}\times [0,1]} \alpha
\varphi(X_{s^{-}}, \alpha,\xi)M(ds,d\alpha,d\xi).
\end{equation} 
Here, $M$ is a Poisson point process
over $\mathbb{R_{+}}\times\mathbb{R}\times [0,1]$ with intensity 
$ds\ \nu(d\alpha)\ d\xi$.
$\nu(d\alpha)$ is a $\sigma$-finite measure on $\mathbb R$ describing the
distribution of new mutations up to a multiplicative constant, 
which satisfies
\begin{equation} \label{eq:condition}
\int_{\mathbb{R}}|\alpha|\wedge 1\nu(d\alpha)<\infty,
\end{equation}
and
\begin{equation*} 
\varphi(x,\alpha,\xi)=\mathbf{1}_{\{\xi\leq g(x,\alpha)\}},
\end{equation*}
where the fixation probability $g(x,\alpha)$ has been defined above.
The points of this Poisson point process $(T_{i},A_{i}, \Xi_{i})$ are such that the
$(T_{i}, A_{i})$ form a Poisson point process over $\mathbb R_+\times\mathbb{R}$  of the
proposed mutations with intensity $ds\nu(d\alpha)$, and the
$\Xi_{i}$ are i.i.d.\ $\mathcal{U}[0,1]$, globally independent of the
Poisson point process of the $(T_{i}, A_{i})$. $T_{i}$'s are the times when mutations
are proposed and $A_{i}$'s are the effect sizes of those mutations. 
The $\Xi_i$ are auxiliary variables determining fixation:
a mutation gets instantaneously fixed if $\Xi_{i}\leq g(X_{T_{i}},A_{i})$,
and is lost otherwise. 
%

\section{Existence and uniqueness}\label{sec3}
Define for all $x$
\begin{align} 
m(x)&=\int_{\mathbb{R}}\alpha g(x,\alpha)\nu(d\alpha),\notag\\
m&=\int_{\mathbb{R}_{+}}\alpha\nu(d\alpha), \label{eq:m}\\  
\psi(x)&=m(x)-v,\label{eq:psi}\\  
V(x)&=\int_{\mathbb{R}}\alpha^{2}g(x,\alpha)\nu(d\alpha),\notag\\
V&=\int_{\mathbb{R}_{+}}\alpha^{2}\nu(d\alpha).\label{eq:V}
\end{align}
$m(x)$ is the mean speed towards zero induced by the fixation of
random mutations while $X_{t}=x<0$. $V(x)$ is related to the second
moment of the distribution of these mutations. $m$ and $V$ are the
limits of $m(x)$ and $V(x)$, respectively, in the case that all
mutations with  $\alpha > 0$ go to fixation (as we will show later,
this is the case if $x\rightarrow-\infty$). Note that our assumptions
do not exclude cases where $m=\infty$ and/or $V=\infty$, unless
stated otherwise. However, since $g(x,\cdot)$ has compact support, for
each $x$, $m(x)<\infty$ and $V(x)<\infty$. The cases $m=\infty$ and
$V=\infty$ correspond to a heavy tailed $\nu$. It would be quite
acceptable on biological grounds to assume that $m<\infty$ and/or
$V<\infty$. However, we refrain from adding unnecessary assumptions.
We rewrite the SDE~\eqref{eq:Xt} as follows
\begin{equation}\label{eq:Xt1}
X_{t}=X_{0}+\int_{0}^{t}\psi(X_{s})ds+\mathcal M_{t}
\end{equation}
where the martingale
\begin{equation}
    \label{eq:martingale}
    \mathcal{M}_{t}=\int_{0}^{t}\int_{\mathbb{R}_{+}}\int_{0}^{1}\alpha\varphi(X_{s^{-}},
    \alpha, \xi) \bar{M}(ds, d\alpha, d\xi),
\end{equation}
with $\bar{M}(ds, d\alpha, \xi)$ being the compensated Poisson measure
$M(ds,d\alpha,d\xi)-ds\nu(d\alpha)d\xi$.
\bigskip
\begin{prop}
Equation~\eqref{eq:Xt1} has a unique solution.
\end{prop}
\begin{proof}
If $\nu$ is a finite measure, then $M$ has a.s.\ finitely many points
in $[0,t]\times\mathbb{R}$ for any $t>0$. In that case, the unique
solution is constructed explicitly by adding the successive jumps. In
the general case, we choose an arbitrary compact set $K=[-k,k]$ (with
$k>0$). There are finitely many jumps $(t_{i},\alpha_{i})$ of $M$ with
$\alpha_{i}\notin K$. It suffices to prove existence and uniqueness
between two such consecutive jumps. In other words , it suffices to
prove existence and uniqueness under the assumption $\nu(K^{c})=0$,
and hence from~\eqref{eq:condition}, we deduce that $\int
\left(|\alpha|+\alpha^2\right)\nu(d\alpha)<\infty$, which we assume from now on. Define for
each $t>0$
\begin{equation}\label{eq:pointfixe}
\Gamma_{t}(U)=x-vt+\int_{[0,t]\times K \times
[0,1]}\alpha\varphi(U_{s^{-}},\alpha,\xi)M(ds,d\alpha,d\xi).
\end{equation}
A solution of equation (\ref{eq:Xt1}) is a fixed point of the mapping $\Gamma$. Hence it suffices to prove that $\Gamma$ admits a unique fixed point. For $\lambda>0$,
{\footnotesize
\begin{align*}
\mathbb{E} e^{-\lambda t}\left|\Gamma_{t}(U)-\Gamma_{t}(V)\right|
=&-\lambda\mathbb{E}\int_{0}^{t}e^{-\lambda s}\left|\Gamma_{s}(U)-\Gamma_{s}(V)\right|ds +\mathbb{E}\int_{0}^{t}e^{-\lambda s}d\left|\Gamma_{s}(U)-\Gamma_{s}(V)\right|\\
\leq &-\lambda\mathbb{E}\int_{0}^{t}e^{-\lambda s}\left|\Gamma_{s}(U)-\Gamma_{s}(V)\right|ds\\
&+\mathbb{E}\int_{[0,t]\times K\times [0,1]}|\alpha|e^{-\lambda s}\left|\varphi(U_{s^{-}},\alpha,\xi)-\varphi(V_{s^{-}},\alpha,\xi)\right|M(ds,d\alpha,d\xi).
\end{align*}
}
The above inequality follows readily from the fact that, for all $0<s<t$, 
\begin{equation*}
\begin{split}
\left|\Gamma_{t}(U)-\Gamma_{t}(V)\right|&-\left|\Gamma_{s}(U)-\Gamma_{s}(V)\right|\\
&\leq \int_{(s,t)\times K\times
[0,1]}|\alpha|\times|\varphi(U_{r^{-}},\alpha,\xi)-\varphi(V_{r^{-}},\alpha,\xi)|M(dr,d\alpha,d\xi).
\end{split}
\end{equation*}
Thus,
{\footnotesize
\begin{equation}
\begin{split}
\label{eq:fctlambda}
\lambda \mathbb{E}\int_{0}^{t}e^{-\lambda s}\left|\Gamma_{s}(U)-\Gamma_{s}(V)\right|ds&\leq\mathbb{E}\int_0^t\int_K |\alpha | e^{-\lambda t}\left|g(U_{s},\alpha)-g(V_{s},\alpha)\right|\nu(d\alpha)ds\\
&\leq c_K \mathbb E\int_0^t e^{-\lambda s}|U_s-V_s|ds.\\
\end{split}
\end{equation}
}
The last inequality is due to the assumption~\eqref{eq:cond3}. Let $T$ be arbitrary. 
Define for all $\lambda>0$ the norm on the Banach space
$L^{1}(\Omega\times [0,T])$, 
\begin{equation*}
\|Z\|_{T,\lambda}=\mathbb{E}\int_{0}^{T}e^{-\lambda t}|Z_{t}|dt.
\end{equation*}
We choose $\lambda_{0}>c_K$. We deduce from~\eqref{eq:fctlambda} that 
\begin{equation*}
\mathbb{E}\|\Gamma(U)-\Gamma(V)\|_{T,\lambda_{0}}\leq \frac{c_K}{\lambda_{0}}
\mathbb{E}\|U-V\|_{T,\lambda_{0}}.
\end{equation*}
Since $c_K/\lambda_{0}<1$, $\Gamma$ has a unique fixed point such that $\Gamma_{t}(U)=U_{t}$ for all $0\leq t\leq T$. Since $T$ is arbitrary, the result is proved.
\end{proof}
We now prove that $g$ given by~\eqref{eq:pfix} and~\eqref{eq:selcoef} satisfies the assumption~\eqref{eq:cond3}.
\bigskip
\begin{lem}
\label{lemcond3}
For any compact set $K\subset\R$ and for all $u,v\in\mathbb R$,
\begin{equation*}
\int_K |\alpha(g(u,\alpha)-g(v,\alpha))|\nu(d\alpha)\leq c_K |u-v|,
\end{equation*}
where $c_K=4\sigma\left(\int_K\alpha^{2}\nu(d\alpha)\right)$.
\end{lem}
\begin{proof}
For $0<u<v$ we have that
\begin{equation*}
\begin{split}
\int_K |\alpha(g(u,\alpha)-g(v,\alpha))|\nu(d\alpha)=
\ &\int_{\mathbb{R_{-}}\bigcap K}|\alpha\left(e^{-2\sigma|\alpha|(2|v|-|\alpha|)^{+}}-e^{-2\sigma|\alpha|(2|u|-|\alpha|)^{+}}\right)|\nu(d\alpha)\\
= &\int_{[-2u,0]\bigcap K}|\alpha\left(e^{-2\sigma|\alpha|(2|v|-|\alpha|)}-e^{-2\sigma|\alpha|(2|u|-|\alpha|)}\right)|\nu(d\alpha)\\
&+\int_{[-2v,-2u]\bigcap K}|\alpha|\left| e^{-2\sigma|\alpha|(2|v|-|\alpha|))}-1\right|\nu(d\alpha)\\
\leq\ &4\sigma\left(\int_{[-2u,0]\bigcap K}\alpha^{2}\nu(d\alpha)\right)\times|u-v|\\
&+2\sigma \int_{[-2v,-2u]\bigcap K}\alpha^{2}(2v+\alpha)\nu(d\alpha)\\
\leq\ &4\sigma\left(\int_K\alpha^{2}\nu(d\alpha)\right)\times|u-v|.
\end{split}
\end{equation*}
A similar estimate can easily be obtained for $v<u<0$.
For $u<0<v$ , we have that
{\footnotesize
\begin{equation*}
\begin{split}
\int_{\mathbb{R}}|\alpha (g(u,\alpha)-g(v,\alpha))|\nu(d\alpha)
\leq\ &\int_{\mathbb{R_{-}}\bigcap K}|\alpha g(v,\alpha)|\nu(d\alpha)
+ \int_{\mathbb{R_{+}}\bigcap K}|\alpha g(u,\alpha)|\nu(d\alpha)\\
\leq\ &2\sigma\int_{\mathbb{R}_{-}\bigcap K}\alpha^{2}(2v+\alpha)^{+}\nu(d\alpha)
+ 2\sigma\int_{\mathbb{R}_{+}\bigcap K}\alpha^{2}(-2u-\alpha)^{+}\nu(d\alpha)\\
\leq\ &4\sigma\int_{\mathbb{R}_{-}\bigcap K}\alpha^{2}|v|\nu(d\alpha)+ 4\sigma\int_{\mathbb{R}_{+}\bigcap K}\alpha^{2}|u|\nu(d\alpha)\\
\leq\ &4\sigma\left(\int_K\alpha^{2}\nu(d\alpha)\right)\times |u-v|.\\
\end{split}
\end{equation*}
}
\end{proof}

\section{Classification of the large-time behaviour}\label{sec4}
\begin{prop}
If $X_{0}>0$, then $X_{t}$ becomes negative after a finite time a.s.
\end{prop}
\begin{proof}[Proof]
Let $T_{0}=\inf(t>0, X_{t}<0)$. Since $g(x,\alpha)=0$ for $x\alpha>0$,
\begin{equation*}
\begin{split}
X_{t\wedge T_{0}^{-}}&=X_{0}-v(t\wedge T_{0})+\int_{0}^{t\wedge
T_{0}^{-}}\int_\mathbb{R}\int_{0}^{1} \alpha \varphi(X_{s^{-}},
\alpha,\xi)M(ds,d\alpha,d\xi)\\
&\leq X_{0}-v(t\wedge T_{0}),
\end{split}
\end{equation*}
hence
\[  t\wedge T_{0}\leq \frac{X_{0}-X_{t\wedge T_{0}^{-}}}{v}< \frac{X_{0}}{v}.\]

Let $t$ tend to $\infty$.
\begin{equation*}
T_{0}\leq \frac{X_{0}}{v}<\infty
\end{equation*}
\end{proof}

	
Whether the process $X_{t}$ is positive recurrent, nul recurrent or
transient depends only upon its behavior while $X_{t}<0$. Hence, we only need to
consider in detail the case when $X_{t}$ is negative, in which case only positive
 mutations ($\alpha>0$) have a positive probability of fixation.
\bigskip
\begin{prop}
The functions $x\longmapsto m(x)$ and $x\longmapsto V(x)$ are continuous and
 decreasing on $\mathbb{R_{-}}$ and 
\begin{equation} \label{eq:prop2} 
\begin{split}
m(x)&\xrightarrow[x\rightarrow -\infty]{} m,\\
V(x)&\xrightarrow[x\rightarrow -\infty]{} V.\\
\end{split}
\end{equation}
\end{prop}
\begin{proof}[Proof]
We prove this result for the function $x\longmapsto m(x)$. A similar 
argument applies to $V(x)$. Let
\begin{align*}
h: \mathbb{R}_{-}\times\mathbb{R}_{+} &\to \mathbb{R}_{+} \\
(x,\alpha)  &\mapsto h(x,\alpha)=\alpha g(x,\alpha).\\
\end{align*}
We have that 
$h(x,\cdot)\in L^{1}(\nu)$, and $x\mapsto h(x,\alpha)$ is decreasing.
For each fixed $\alpha>0$, $0\leq h(x,\alpha)\uparrow
\alpha$, as $x\rightarrow -\infty$.
By the monotone convergence theorem, it follows that 
\begin{equation*}
m(x)=\int_{\mathbb{R}}h(x,\alpha)\nu(d\alpha)\xrightarrow[x\rightarrow -\infty]{} m.
\end{equation*}
Continuity is proved similarly. 
\end{proof}

To determine the large-time behavior of the process, we now consider
successively, the three cases $v > m$, $v < m$ and $v=m$.

\subsection{The case $v>m$}
In particular, here $m=\int_{0}^{\infty}\alpha\nu(d\alpha)$ is finite. 
Let 
\begin{equation*}
\mathcal{N}_{t}=\int_{[0,t]\times\mathbb{R}\times [0,1]} \alpha
\varphi(X_{s^{-}}, \alpha,\xi)M(ds,d\alpha,d\xi)
\end{equation*}
be the sum of all the jumps on the time interval $[0,t]$. We have that
\begin{equation*}
\mathcal{N}_{t}= \mathcal{N}_{t}^{(+)}+\mathcal{N}_{t}^{(-)}\leq \mathcal{N}_{t}^{(-)},
\end{equation*} 
where 
\begin{equation*}
\begin{split}
\mathcal{N}_{t}^{(+)}&=\mathbf{1}_{\{X_{s^{-}}>0\}}d\mathcal{N}_{s}\\
\mathcal{N}_{t}^{(-)}&=\mathbf{1}_{\{X_{s^{-}}<0\}}d\mathcal{N}_{s}\\
\end{split}
\end{equation*}
Let $m^{(-)}(x)=\mathbf{1}_{\{x<0\}}m(x)$, hence
\begin{equation*}
\mathcal{M}_{t}^{(-)}=\mathcal N^{(-)}_t - \int_{0}^{t}m^{(-)}(X_s)ds.
\end{equation*}
Thus,
\begin{equation*}
X_{t}\leq X_{0}+\int_{0}^{t}(m^{-}(X_{s})-v)ds+\mathcal{M}_{t}^{(-)}
\end{equation*}
\bigskip
\begin{lem}
If $m<\infty$, then
\begin{equation} \label{eq:martconv}
\lim_{t\rightarrow\infty}\frac{\mathcal M^{(-)}_{t}}{t}=0.
\end{equation}
\end{lem}
\begin{proof}[Proof]
$\mathcal M^{(-)}_{t}$ is a square-integrable martingale, 
such that $\mathbb{E}{\mathcal M^{(-)}_{t}}=0$.
Define $M^{0}$ as the Poisson random measure for new mutations
i.e.\ a Poisson point process on $\mathbb{R}_{+}\times\mathbb{R}$
with intensity $ds\nu(d\alpha)$.  
For all $i \in \mathbb{N}^{*}$ and $n\in \mathbb{N}^{*}$, define
\begin{equation*}
\begin{split}
\xi_{i} &= \int_{i-1}^{i}\int_{0}^{\infty}\int_{0}^{1}
\alpha\varphi(X_{s^{-}},\alpha,\xi)M(ds,d\alpha,d\xi),\\
\omega_{i} &= \int_{i-1}^{i}m^{-}(X_{s})ds,\\
\eta_{i} &= \int_{i-1}^{i}\int_{0}^{\infty} \alpha M^{0}(ds,d\alpha),\\
Y_{i} &= \xi_{i}-\omega_{i},\\
\mathcal{M}_{n}&=\sum_{i=1}^{n}Y_{i},\\
\end{split}
\end{equation*}
Note that for all $i \in \mathbb{N}^{*}$, 
$0\leq \xi_{i}\leq \eta_{i}$ and $0\leq \omega_{i}\leq m$. 
We first establish
\bigskip
\begin{lem}
\begin{equation*}
\mbox{If
}\quad\frac{\sum_{i=1}^{n}Y_{i}}{n}\xrightarrow[n]{}0,\quad\mbox{
then}  \quad\frac{\mathcal M^{-}_{t}}{t}\xrightarrow[t\rightarrow\infty]{}0.
\end{equation*}
\end{lem}
\begin{proof}[Proof]
\begin{equation*}
\frac{\mathcal M^{-}_{t}}{t}=\frac{\mathcal M^{-}_{\lfloor t \rfloor}}{\lfloor t \rfloor}\times
\frac{\lfloor t \rfloor}{t}+\frac{\tilde{\mathcal M^{-}}_{t}}{t},
\end{equation*}
where
\begin{equation*}
\begin{split}
\frac{\tilde{\mathcal M}^{-}_{t}}{t}=&\frac{1}{t}\left(\int_{\lfloor t
\rfloor}^{t}\int\int \alpha \varphi(X_{s^{-}},\alpha,\xi)M(ds,d\alpha,d\xi)-\int_{\lfloor t \rfloor}^{t}m^{-}(X_{s})ds\right)\\
&\leq \frac{1}{t}\left(\int_{\lfloor t \rfloor}^{\lceil t
\rceil}\int\int \alpha \varphi(X_{s^{-}},\alpha,\xi)M(ds,d\alpha,d\xi)+\int_{\lfloor t \rfloor}^{\lceil t \rceil}m^{-}(X_{s})ds\right)\\
&=\frac{1}{t}\left(\xi_{\lceil t \rceil}+\omega_{\lceil t
\rceil}\right)=\frac{\lceil t \rceil}{t}\times \frac{1}{\lceil t
\rceil}\left(Y_{\lceil t \rceil }+2\omega_{\lceil t
\rceil}\right)\xrightarrow[t\rightarrow \infty]{}0,\\
\end{split}
\end{equation*}
since for all $n>0$,
\begin{equation*}
\frac{Y_{n+1}}{n+1}=\frac{\sum_{i=1}^{n+1}Y_{i}}{n+1}-\frac{\sum_{i=1}^{n}Y_{i}}{n}\times\frac{n}{n+1}\xrightarrow[n\rightarrow\infty]{}0
\end{equation*}
and
\begin{equation*} 
0\leq \frac{|\omega_{n}|}{n}\leq \frac{m}{n},
\end{equation*}
hence 
\begin{equation*}
\frac{\omega_{n}}{n}\xrightarrow[n\rightarrow \infty]{}0.
\end{equation*}
\end{proof}
Back to the proof of Lemma 1. We now define 
\begin{equation*}
\begin{split}
A_{i} &=\{\eta_{i}>i\},\\
\tilde{Y}_{i} &=Y_{i}\mathbf{1}_{\{\eta_{i}\leq i\}}.\\
\end{split}
\end{equation*}
Since the $\left(\eta_{i},i\in \mathbb{N}^{*}\right)$ are i.i.d, integrable and
\begin{equation*}
\mathbb{P}(\eta_{i}>i)=\sum_{i\geq 1}\mathbb{P}(\eta_{1}>i)\leq \mathbb{E}\eta_{1}<\infty,
\end{equation*}
it follows from Borel Cantelli's Lemma that $\mathbb{P}(\limsup
A_{i})=0$. Hence, a.s.\ there exists $N(\alpha)$ such that for all
$n>N(\alpha)$, we have
$\tilde{Y}_{n}=Y_{n}$. But since
$\mathbb{E}(\tilde{Y}_{n})\rightarrow \mathbb{E}(Y_{1})$ due to the
dominated convergence theorem, it is sufficient to prove that
\begin{equation*}
\frac{\sum_{i=1}^{n}\left(\tilde{Y}_{i}-\mathbb{E}(\tilde{Y}_{i})\right)}{n}\xrightarrow[n]{}0.
\end{equation*}
Due to corollary 3.22 in~\cite{Bre68}\footnote{In the proof
of this theorem, we replace Kolmogorov's inegality by Doob's
inegality for martingales, and the result holds in our case.}, it is
again sufficient to prove that
\begin{equation*}
\sum_{i=1}^{\infty}\frac{\mathbb{E}(\tilde{Y_{i}}^{2})}{i^{2}}<\infty.
\end{equation*}
Indeed, we have that
\begin{equation*}
\sum_{i=1}^{\infty}\frac{\mathbb{E}(\tilde{Y_{i}}^{2})}{i^{2}}<2m.
\end{equation*}
The underlying calculation can be found in the proof of
theorem 3.30 in~\cite{Bre68}.
\end{proof}
\bigskip
\begin{rmk}
\label{Mplus}
In the case $m<\infty$ and $X_{t}\rightarrow -\infty$, we have that 
\begin{equation*}
\frac{1}{t}\mathcal{M}_{t}^{(+)}\rightarrow 0,
\end{equation*}
since eventually $X_{t}$ becomes negative. Furthermore,
if we assume that $\int_{-\infty}^{0}\alpha \nu(d\alpha)>-\infty$ then the 
previous Lemma implies that 
\begin{equation*}
\frac{\mathcal{M}_{t}}{t}\rightarrow 0,
\end{equation*}
whether $X_{t}\rightarrow -\infty$ or not. But we refrain from adding 
any supplementary assumption on $\nu$.
\end{rmk}
\bigskip
\begin{prop}
In the case $v>m$, $X_{t}\rightarrow - \infty$ with speed $v-m$ in the sense
that
\begin{equation*}
\frac{X_t}{t}\xrightarrow[]{a.s. }m-v \text{ as } t\rightarrow\infty.
\end{equation*}
\end{prop}
\begin{proof}[Proof]
\begin{equation*}
\frac{X_{t}}{t}=\frac{X_{0}}{t}-v+\frac{1}{t}\int_{0}^{t}m(X_{s})ds+\frac{\mathcal M_{t}}{t}\leq
\frac{X_{0}}{t}-v+m+\frac{\mathcal M_{t}}{t}.
\end{equation*}
Hence
\begin{equation}\label{eq:eq0}
\limsup_{t\rightarrow \infty}\frac{X_{t}}{t}\leq-v+m.
\end{equation}
On the other hand, it follows from~\eqref{eq:prop2} that 
\begin{equation} \label{eq:eq1}
\forall \epsilon>0\quad \exists K_{\epsilon}>0 \mbox{ such that }
x\leq -K_{\epsilon}\Rightarrow m(x)>m-\epsilon,
\end{equation} 
and since  $X_{t}\xrightarrow[t\rightarrow\infty]{}-\infty$
by~\eqref{eq:eq0}, we have
\begin{equation} \label{eq:eq2}
\forall \epsilon>0\quad \exists t_{\epsilon}>0 \mbox{ such that } \forall s\geq
t_{\epsilon}\Rightarrow X_{s}\leq -K_\epsilon. 
\end{equation}
Statements~\eqref{eq:eq1} and~\eqref{eq:eq2} combined give
\begin{equation*}
\forall \epsilon >0\quad \exists t_{\epsilon} \mbox{ such that }
\forall s\geq t_{\epsilon} \Rightarrow m(X_{s})>m-\epsilon. 
\end{equation*}
Then, $\forall \epsilon>0$ and $t>t_\epsilon$
\begin{equation*}
\begin{split}
\frac{X_{t}}{t}&\geq\frac{X_{t_{\epsilon}}}{t}+\frac{1}{t}\int_{t_{\epsilon}}^{t}(m(X_{s})-v)ds+\frac{\mathcal
M_{t}-\mathcal M_{t_{\epsilon}}}{t}\\
&\geq\frac{X_{t_{\epsilon}}}{t}+(m-\epsilon-v)\times\frac{t-t_\epsilon}{t}+\frac{\mathcal
M_{t}-\mathcal M_{t_{\epsilon}}}{t}.
\end{split}
\end{equation*}
Hence,
\begin{equation*}
\liminf_{t\rightarrow\infty}\frac{X_{t}}{t}\geq -v+m
\end{equation*}

We conclude that $X_{t}\rightarrow-\infty$ with speed $v-m$.
\end{proof}

\subsection{The case $v<m$}
The assumption is satisfied, in particular, when $m=\infty$.
\bigskip
\begin{prop}\label{reccase}
In the case $v<m$, $X_{t}$ is positive recurrent. 
\end{prop}
\begin{proof}[Proof]
Since $x\longmapsto m(x)$ is continuous, decreasing from 
$\mathbb{R}_{-}$ to $\mathbb{R}_{+}$ and $m(0)=0<v<m$, $\exists$
$N>0$ such that $m(-N)=v$. We choose an arbitrary $K>N$, so that for all $x<-K$
\begin{equation*}
\psi(x)>m(-K)-v>0.
\end{equation*} 

Assume that $X_{0}<-K$, and define the stopping time
\begin{equation*}
T_{K}=\inf\{t>0, X_{t}\geq -K\}
\end{equation*}
We have
\begin{equation*}
\begin{split}
X_{t\wedge T_{K}}&=X_{0}+\int_{0}^{t\wedge
T_{K}}\psi(X_{s})ds+\int_{[0,t\wedge T_{K}]\times \mathbb{R}_{+}\times
[0,1]}\alpha\varphi(X_{s^{-}}, \alpha, \xi)\bar{M}(ds, d\alpha, d\xi)\\
&>X_{0}+(m(-K)-v)(t\wedge T_{K})+\mathcal M_{t\wedge T_{K}}.
\end{split}
\end{equation*} 
Thus,
\begin{equation*}
0>\mathbb{E}(X_{t\wedge T_{K}})>X_{0}+(m(-K)-v)\mathbb{E}(t\wedge
T_{K}).
\end{equation*}
Now let $t$ tend to $\infty$. It follows by monotone convergence that
\begin{equation} \label{eq:ineq}
\mathbb{E}(T_{K})<\frac{-X_{0}}{m(-K)-v}<\infty.
\end{equation}
Given any fixed $T>0$, let $p$ denote the lower bound of the
probability that, starting from any given point $x\in[-K,0)$ at time
$t_{0}$, $X$ hits $[0,+\infty)$ before time $T$. Clearly $p>0$. We now define a geometric random variable $\beta$ with success probability $p$. Let us restart our process $X$ 
at time $t_{0}=T_{K}$ from $x_{0}\in[-K,0)$. If
$X$ hits zero before time $T$, then $\beta=1$. If not, we look at the
position $X_{T}$ of $X$ at time $T$. Two cases are possible:
\begin{itemize}
\item If $X_{T}<-K$, we wait until $X$ goes back above $-K$. Since
$X_{T}\geq -(K+vT)$, the time $\alpha_{2}$ needed to do so
satisfies 
\begin{equation*}
\mathbb{E}(\alpha_{2})\leq \frac{K+vT}{m(-K)-v}.
\end{equation*}
This calculation is similar to~\eqref{eq:ineq}.
\item If $X_{T}\geq -K$, we start afresh from there, since the
probability to reach zero in less than $T$ is greater than or equal to
$p$. 
\end{itemize}
So either at time $T$ or at time $T+\alpha_{2}$, we start
again from a level which is above $-K$. If $[0,+\infty)$ is reached
during the next time interval of length $T$, then $\beta=2$. If not,
we repeat the procedure. A.s.\ one of the mutually independent 
trials is successful. We have that  
\begin{equation*}
T_{0}<T_{K}+\sum_{i=1}^{\beta}\left(T+ \alpha_{i}\right),
\end{equation*}
where the random variables $\left(\alpha_{i}\right)_{i}$ are i.i.d, globally independent of $\beta$. Hence 
\begin{equation*}
\mathbb{E}T_{0}<\mathbb{E}T_{K}+\frac{1}{p}\left(T+\frac{K+vT}{m(-K)-v})\right),
\end{equation*}
and the process is positive recurrent.
\end{proof}

\subsection{The case $v=m$}\label{m=v}
We first state a lemma that we will apply several times in this
section. 
\bigskip
\begin{lem}\label{ito-ineq}
Let $X_{t}$ be a FV c\`adl\`ag process.
\begin{enumerate}
\item  If $\Phi\in C^{1}$, then
\begin{equation*}
\Phi(X_{t})=\Phi(X_{0})+\int_{0}^{t}\Phi'(X_{s^{-}})dX_{s}+\sum_{s\leq t, \Delta X_{s}\neq 0} \Phi(X_{s^{-}}+\Delta X_{s})-\Phi(X_{s^{-}})-\Phi'(X_{s^{-}})\Delta X_{s},
\end{equation*}
where $\Delta X_{s}=X_{s}-X_{s^{-}}$, $\forall s$.
\item Moreover, if $\Phi\in C^{2}$ such that $\Phi''$ is an increasing function and $\Delta X_{s}\geq 0$ for all $s$, then
\begin{equation*}
\Phi(X_{t})-\Phi(X_{0})-\int_{0}^{t}\Phi'(X_{s^{-}})dX_{s}\geq\frac{1}{2}\sum_{s\leq t, \Delta X_{s}\neq 0} \Phi''(X_{s}) (\Delta X_{s})^{2}.
\end{equation*}
If $\Phi\in C^{2}$ such that $\Phi''$ is a decreasing function and $\Delta X_{s}\geq 0$ for all $s$, then
\begin{equation*}
\Phi(X_{t})-\Phi(X_{0})-\int_{0}^{t}\Phi'(X_{s^{-}})dX_{s}\leq\frac{1}{2}\sum_{s\leq t, \Delta X_{s}\neq 0} \Phi''(X_{s^{-}})(\Delta X_{s})^{2}.
\end{equation*}
In particular, choosing $\Phi(x)=x^{2}$, we deduce that 
\begin{equation}
X_{t}^{2}=X_{0}^{2}+2\int_{0}^{t}X_{s^{-}}dX_{s}+\sum_{s\leq t}\left(\Delta X_{s}\right)^{2}.
\end{equation}
\end{enumerate}
\end{lem}

\begin{proof}[Proof]
The first part of this lemma is a well known result (see
\citealp{Pro05}). We will only prove part 2 of the lemma. If $\Phi\in
C^{2}$ then it follows from Taylor's formula that there exists a
random function $\beta$ taking its values in $[0,1]$ such that for all
$s$
\begin{equation*}
\Phi(X_{s})-\Phi(X_{s^{-}})-\Phi'(X_{s^{-}})\Delta X_{s}=\frac{1}{2}\Phi''(X_{s^{-}}+\beta_{s}\Delta X_{s})\left(\Delta X_{s}\right)^{2}.
\end{equation*}
If $\Phi''$ is an increasing function and $y\geq 0$ then 
\begin{equation*}
\Phi''(x)\leq\Phi''(x+\beta_{s}y)\leq \Phi''(x+y).
\end{equation*}
If $\Phi''$ is a decreasing function and $y\geq 0$ then
\begin{equation*}
\Phi''(x+y)\leq\Phi''(x+\beta_{s}y)\leq \Phi''(x).
\end{equation*}
\end{proof}
Note that $V\leq \infty$ and at this stage we do not assume that $V$
is finite. In the case $m=v$, the asymptotic behavior of the process
$X_{t}$ depends on the asymptotic behavior of the mean net rate of
adaptation $\psi(x)$ defined in~\eqref{eq:psi} as $x\rightarrow
-\infty$.

\bigskip
\begin{prop}\label{prop1}
We assume that $m=v$. If moreover
\begin{equation}\label{eq:cond1}
\limsup_{x\rightarrow -\infty}\left|x\psi(x)\right|<\frac{V}{2},
\end{equation}
then the process $X_t$ is null recurrent. 
\end{prop}
We first establish 
\bigskip
\begin{lem}\label{lem:to0}
Under the condition $m<\infty$, we have that
\[ \lim_{x\to-\infty}\frac{V(x)}{|x|}=0.\] 
\end{lem}
\begin{proof}
Consider, for $x<x_0<0$,
\begin{align*}
\frac{V(x)}{|x|}&\le\frac{1}{|x|}\int_0^{2|x|}\alpha^2\nu(d\alpha)\\
&=\frac{1}{|x|}\int_0^{2|x_0|}\alpha^2\nu(d\alpha)+\frac{1}{|x|}\int_{2|x_0|}^{2|x|}\alpha^2\nu(d\alpha)\\
&\le\frac{1}{|x|}\int_0^{2|x_0|}\alpha^2\nu(d\alpha)+2\int_{2|x_0|}^{\infty}\alpha\nu(d\alpha),
\end{align*}
hence
\[\limsup_{x\to-\infty}\frac{V(x)}{|x|}\le2\int_{2|x_0|}^{\infty}\alpha\nu(d\alpha),\]
and our condition implies that the last right hand side tends to 0, as $x_0\to-\infty$. The results follows.
\end{proof}
We can now return to the
\begin{proof}[Proof of Proposition \ref{prop1}]
First note that, since $m=v$ implies $\psi(x)\leq 0$ for all $x\leq 0$, 
condition~\eqref{eq:cond1} is equivalent to
\begin{equation*}
\liminf_{x\rightarrow -\infty}|x|\psi(x)>-\frac{V}{2}.
\end{equation*}
To prove recurrence under condition~\eqref{eq:cond1}, we recall that 
\begin{equation} \label{eq:process}
X_{t}=X_{0}+\int_{0}^{t}\psi(X_{s})ds+\mathcal M_{t}.
\end{equation}
We will apply Lemma 3 with $f(x)=\log|x|$, with $x<0$. Here $f''$ is decreasing. Hence as long as $X_{t}$ remains negative,
\begin{equation*}
\begin{split}
\log|X_{t}|&\leq\log|X_{0}|+\int_{0}^{t}\frac{\psi(X_{s})}{X_{s}}ds+\int_{0}^{t}\frac{1}{X_{s^{-}}}d\mathcal M_{s}-\frac{1}{2}\sum_{s\leq t}\frac{(\Delta X_{s})^{2}}{X^{2}_{s^{-}}}\\
&\leq\log|X_{0}|+\int_{0}^{t}\frac{\psi(X_{s})}{X_{s}}ds+\int_{0}^{t}\frac{1}{X_{s^{-}}}d\mathcal M_{s}\\
&-\frac{1}{2}\int_{0}^{t}\int_{\mathbb{R}_{+}}\int_{0}^{1}\frac{\alpha^{2}\varphi(X_{s^{-}},\alpha,\xi)}{X_{s^{-}}^{2}}\bar{M}(ds,d\alpha,d\xi)-\frac{1}{2}\int_{0}^{t}\frac{V(X_{s})}{X_{s}^{2}}ds\\
&=\log|X_{0}|+\int_{0}^{t}\left(\frac{\psi(X_{s})}{X_{s}}-\frac{V(X_{s})}{2X_{s}^{2}}\right)ds+\hat{\mathcal{M}}_{t},\\
\end{split}
\end{equation*}
where $\hat{\mathcal{M}}$ is a martingale. For all $a<b<0$, define the stopping time
\begin{equation*}
S_{a,b}=\inf(t>0,X_{t}\leq a \mbox{ or } X_{t}\geq b).
\end{equation*}
It follows from our assumption that there exists $L>0$ such that 
\begin{equation}\label{eq:infL}
\inf_{x\leq -L}\left(|x|\psi(x)+\frac{V(x)}{2}\right)>0.
\end{equation}
For any $N>L$, from Doob's optional sampling theorem, if $-N<X_0<L$,
\begin{equation*}
\begin{split}
\mathbb{E}\log|X_{t\wedge S_{-N,-L}}|&\leq \log|X_{0}|+\mathbb{E}\int_{0}^{t\wedge S_{-N,-L}}\left(\frac{\psi(X_{s})}{X_{s}}-\frac{V(X_{s})}{2X_{s}^{2}}\right)ds.\\
\end{split}
\end{equation*}
Letting $t$ tend to $\infty$,  
\begin{equation*}
\mathbb{E}\log|X_{S_{-N,-L}}|\leq\log|X_{0}|.
\end{equation*}
Define the stopping times
\begin{equation*}
\begin{split}
T_{-L}^{\uparrow}&=\inf(t>0, X_{t}\geq -L),\\
T_{-N}^{\downarrow}&=\inf(t>0, X_{t}\leq -N).
\end{split}
\end{equation*}
It follows from the previous estimate that 
\begin{equation*}
\log N \times\mathbb{P}(T_{-N}^{\downarrow}<T_{-L}^{\uparrow})<\log|X_{0}|.
\end{equation*}
We deduce that $\mathbb{P}(T_{-N}^{\downarrow}<T_{-L}^{\uparrow})\rightarrow 0$ 
as $N$ tend to $\infty$. We conclude that the process returns 
a.s.\ an infinite number of times above $-L$, hence also above 
$0$ by a classical argument (see the proof of Proposition~\ref{reccase}). Therefore, the process $X$ is recurrent.

Let now $X_{0}<-(L+1)$. For all $N>L$, multiplying~\eqref{eq:process} by $-1$, we have
\begin{equation*}
|X_{t\wedge S_{-N,-L}}|=|X_{0}|-\int_{0}^{t\wedge
S_{-N,-L}}\psi(X_{s})ds-\int_{0}^{t\wedge S_{-N,-L}}d\mathcal{M}_{s},
\end{equation*}
By Doob's theorem and letting $t$ tend to $\infty$, since again $\psi(x)\leq 0$ for $x\leq 0$
\begin{equation*}
\begin{split}
\mathbb{E}|X_{S_{-N,-L}}|= |X_{0}|-\mathbb{E}\int_{0}^{ S_{-N,-L}}\psi(X_{s})ds&\geq |X_{0}|\mbox{, hence}\\
L\mathbb{P}(T_{-L}^{\uparrow}<T_{-N}^{\downarrow})+N\mathbb{P}(T_{-N}^{\downarrow}<T_{-L}^{\uparrow})&\geq |X_{0}|.\\
\end{split}
\end{equation*}
We have 
\begin{equation} \label{eq:limit}
\liminf_{N\rightarrow\infty} N\mathbb{P}(T_{-N}^{\downarrow}<T_{-L}^{\uparrow})\geq |X_{0}|-L>0.
\end{equation}
It follows from Lemma 3 that 
\begin{equation*}
X_{t}^{2}=X_{0}^{2}-\int_{0}^{t}2|X_{s}|\psi(X_{s})ds+\int_{0}^{t}2X_{s^{-}}d\mathcal{M}_{s}+\sum_{s\leq t}(\Delta X_{s})^{2}.
\end{equation*}
On the other hand,
\begin{equation*}
\begin{split}
\sum_{s\leq t}(\Delta X_{s})^{2}&=
\int_{0}^{t}\int_{\mathbb{R_{+}}}\int_{0}^{1}\alpha^{2}\varphi(X_{s^{-}},\alpha,\xi)\bar{M}(ds,d\alpha,d\xi)\\
&+\int_{0}^{t}\int_{\mathbb{R_{+}}}\alpha^{2}g(X_{s^{-}},\alpha)\nu(d\alpha)ds.\\
\end{split}
\end{equation*}
Thus, from~\eqref{eq:infL} and the monotonicity of $V(x)$
\begin{equation*}
X_{t\wedge S_{-N,-L}}^{2}\leq X_{0}^{2}+\int_{0}^{t\wedge S_{-N,-L}} 2V(-N) ds+\tilde{\mathcal{M}}_{t\wedge S_{-N,-L}},
\end{equation*}
where $\tilde{\mathcal{M}}_{\cdot\wedge S_{-N,-L}}$ is a martingale.
Letting $t$ tend to $\infty$, we have for all $\epsilon>0$
\begin{equation*}
\begin{split}
\mathbb{E}X_{S_{-N,-L}}^{2}&\leq X_{0}^{2}+2V(-N)\mathbb{E}S_{-N,-L}\text{, hence}\\
\mathbb{E}S_{-N,-L}&\geq\frac{L^{2}\mathbb{P}(T_{-L}^{\uparrow}<T_{-N}^{\downarrow})+N^{2}\mathbb{P}(T_{-N}^{\downarrow}<T_{-L}^{\uparrow})-X_{0}^{2}}{2V(-N)}.\\
\end{split}
\end{equation*}
It follows by monotone convergence that 
\begin{equation*}
\mathbb{E}(T_{-L}^{\uparrow})=\lim_{N\rightarrow\infty}\mathbb{E}S_{-N,-L}\geq \liminf_{N\rightarrow\infty}\left\{N\mathbb{P}(T_{-N}^{\downarrow}<T_{-L}^{\uparrow})\times \frac{N}{2V(-N)}-\frac{X_{0}^{2}}{2V(-N)}\right\}.
\end{equation*}
Combining this with Lemma \ref{lem:to0} and~\eqref{eq:limit}, we
deduce that $\mathbb{E}T_{-L}^{\uparrow}=\infty$ and the process is
null recurrent.
\end{proof}
\begin{rmk}
Condition~\eqref{eq:cond1} is rather weak. It is satisfied as soon as both the measure $\nu$
and $V$ are finite. We give the proof below. 
It is also satisfied for some measures that don't have a second moment such as 
\begin{equation*}
\nu(d\alpha)\approx\frac{d\alpha}{\alpha^{2+\delta}}\mathbf 1_{\{\alpha\geq 1\}}, \quad \frac{1}{2}<\delta\leq 1.
\end{equation*}
\end{rmk}
\bigskip
\begin{prop}
\label{propbetween}
If $\nu$ is a finite measure, $V<\infty$ and the fixation probability is given by~\eqref{eq:pfix}
and~\eqref{eq:selcoef}, then \eqref{eq:cond1} is satisfied. 
\end{prop}
\begin{proof}
Let for all $x$,
\begin{equation*}
\begin{split}
D(x)=|x\psi(x)|-\frac{V(x)}{2}&=|x|\int_{2|x|}^{\infty}\alpha\nu(d\alpha)-\int_0^{2|x|}\frac{\alpha^2}{2}\nu(d\alpha)\\
&+\int_0^{2|x|}\left(|x|\alpha+\frac{\alpha^2}{2}\right)e^{-2\sigma \alpha\left(2|x|-\alpha\right)}\nu(d\alpha).\\
\end{split}
\end{equation*}
It follows by dominated convergence that
\begin{equation*}
\frac{1}{2}\int_0^{\infty}\alpha^2 e^{-2\sigma \alpha\left(2|x|-\alpha\right)}\mathbf 1_{[0,2|x|]}\nu(d\alpha)\xrightarrow[x\rightarrow -\infty]{}0,
\end{equation*}
since , 
\begin{equation*}
\begin{split}
&\alpha^2 e^{-2\sigma \alpha\left(2|x|-\alpha\right)}\mathbf 1_{[0,2|x|]}\xrightarrow[x\rightarrow -\infty]{a.s.}0 \text{ for all } \alpha>0,\\
\text{and }&\alpha^2 e^{-2\sigma \alpha\left(2|x|-\alpha\right)}\mathbf 1_{[0,2|x|]}\leq \alpha^2\in L^1(\nu).\\
\end{split}
\end{equation*}
On the other hand, 
\begin{equation*}
\int_0^{2|x|}|x|\alpha e^{-2\sigma \alpha\left(2|x|-\alpha\right)}\nu(d\alpha)=\int_0^{|x|}|x|\alpha e^{-2\sigma \alpha\left(2|x|-\alpha\right)}\nu(d\alpha)
+\int_{|x|}^{2|x|}|x|\alpha e^{-2\sigma \alpha\left(2|x|-\alpha\right)}\nu(d\alpha).
\end{equation*}
Note that if $0\leq \alpha \leq |x|$ then $2|x|-\alpha\geq |x|$, thus
\begin{equation*}
e^{-2\sigma \alpha(2|x|-\alpha)}\leq e^{-2\sigma \alpha|x|}.
\end{equation*}
In addition the function 
\begin{equation*}
\begin{array}{cccc}
f_{\sigma}:& \R_+ &\rightarrow & \R_+\\
& z & \rightarrow & f_{\sigma}(z)=ze^{-\sigma z}
\end{array}
\end{equation*}
has a maximum for $z=\frac{1}{\sigma}$. It follows that 
\begin{equation*}
\int_0^{|x|}|x|\alpha e^{-2\sigma \alpha\left(2|x|-\alpha\right)}\nu(d\alpha)\leq \frac{2 }{\sigma e}\int_0^{\infty}e^{-\sigma |x|\alpha}\nu(d\alpha)\xrightarrow[x\rightarrow -\infty]{}0,
\end{equation*}
again by dominated convergence.
The second term also goes to 0 when $x\rightarrow -\infty$. In fact,
\begin{equation*}
\int_{|x|}^{2|x|}|x|\alpha e^{-2\sigma \alpha\left(2|x|-\alpha\right)}\nu(d\alpha)\leq \int_{0}^{\infty}\alpha^2 e^{-2\sigma \alpha\left(2|x|-\alpha\right)}\nu(d\alpha),
\end{equation*} 
and by the same argument as before the result follows since
\begin{equation*}
\begin{split}
&\alpha^2 e^{-2\sigma \alpha\left(2|x|-\alpha\right)}\mathbf 1_{[|x|,2|x|]}\xrightarrow[x\rightarrow -\infty]{a.s.}0 \text{ for all } \alpha>0,\\
\text{and }&\alpha^2 e^{-2\sigma \alpha\left(2|x|-\alpha\right)}\mathbf 1_{[|x|,2|x|]}\leq \alpha^2\in L^1(\nu).
\end{split}
\end{equation*} 
Furthermore, it follows from the fact that $V$ is finite that
\begin{equation*}
|x|\int_{2|x|}^{\infty}\alpha \nu(d\alpha)\leq \int_{2|x|}^{\infty}\frac{\alpha^2}{2} \nu(d\alpha)\xrightarrow[x\rightarrow -\infty]{}0.
\end{equation*}
Hence,
\begin{equation*}
\limsup_{x\rightarrow -\infty}D(x)=-\frac{V}{2}<0.
\end{equation*}
\end{proof}

We now consider the case $m=v$ and $\liminf\limits_{x\rightarrow
-\infty}\left|x\psi(x)\right|>\frac{V}{2} $, which implies in
particular that $V<\infty$.
\bigskip
\begin{prop}\label{prop2}
Assume that $m=v$ and
\begin{equation}\label{eq:cond2}
\liminf_{x\rightarrow -\infty}\left|x\psi(x)\right|>\frac{V}{2}.
\end{equation}
If, moreover, there exist $0<p_{0}<1$ and $0<\beta_{0}<1$
such that for all $0<\beta<\beta_{0}$
\begin{equation} \label{eq:prop6}
|x|^{p_{0}+2}\int_{-\beta x}^{\infty}\alpha^{2}g(x,\alpha)\nu(d\alpha)\xrightarrow[x\rightarrow -\infty]{}0,
\end{equation}
then $X_{t}$ is transient, that is, $X_{t}\rightarrow -\infty$, and moreover $\frac{X_{t}}{t}\rightarrow 0$.
\end{prop} 
\bigskip
\begin{rmk}
The conditions of Proposition~\ref{prop2} are satisfied in the case where both $\nu$
is infinite and its tail is thin enough, while $g$ is given by~\eqref{eq:pfix} and \eqref{eq:selcoef}. For example, if
\begin{equation*}
\nu(d\alpha)=\left(\frac{1}{\alpha^{1+\delta}}\mathbf{1}_{|\alpha|<1}+\rho(\alpha)\mathbf{1}_{|\alpha|>1}\right)d\alpha,
\end{equation*}
where $\rho(\alpha)\leq C|\alpha|^{-(5+\delta')}, \: |\alpha|>1$ for some $\delta,\delta'>0$. 
Condition~\eqref{eq:cond2} follows from the fact that $V<\infty$ while $|x\psi(x)|\rightarrow \infty$
as $|x|\rightarrow\infty$, since, cf. proof of Proposition~\ref{propbetween},
\begin{equation*}
\begin{split}
|x|\int_0^1 \alpha e^{-2\sigma\alpha(2|x|-\alpha)}\nu(d\alpha)
&\geq |x|\int_0^1\alpha^{-\delta} e^{-4\sigma\alpha|x|}d\alpha\\
&=|x|^\delta \int_0^{|x|}e^{-4\sigma z} \frac{dz}{z^\delta}.\\
\end{split}
\end{equation*}
Condition~\eqref{eq:prop6} is easy to check.
\end{rmk}
\bigskip
\begin{proof}[Proof]
First note that condition~\eqref{eq:cond2} is equivalent to
\begin{equation*}
\limsup_{x\rightarrow -\infty}|x|\psi(x)<-\frac{V}{2}.
\end{equation*}
Hence there exist $K>0$ and $0<p\leq p_{0}$ such that
\begin{equation}\label{eq:condsup}
\sup_{x\leq -K}\left(|x|\psi(x)+(2p+1)\frac{V(x)}{2}\right)<0.
\end{equation}
Let $f$ be the $C^{2}(\mathbb{R})$-function such that $f(-1) = 1$, $f'(-1)=p$, and
\begin{equation*}
f''(x)=\frac{p(p+1)}{|x|^{p+2}}\mathbf{1}_{\{x\leq -1\}}+p(p+1) \mathbf{1}_{\{x\geq -1\}},
\end{equation*}
with $p$ being a real number in $(0,1)$ for which~\eqref{eq:condsup} holds.
Then it follows from Lemma 3 applied to $f$, since $f''$ is an
increasing function,
\begin{equation*}
f(X_{t})\leq
f(X_{0})+\int_{0}^{t}\psi(X_{s})f'(X_{s})ds+\frac{1}{2}\int_{0}^{t}\int_{0}^{\infty}f''(X_{s}+\alpha)\alpha^{2}g(X_{s},\alpha)\nu(d\alpha)ds+\mathcal{N}_{t},
\end{equation*}
where the martingale $\mathcal{N}$ is defined by
\begin{equation*}
\mathcal{N}_{t}=\frac{1}{2}\int_{0}^{t}\int_{0}^{\infty}\int_{0}^{1}\left[f'(X_{s^-})+f''(X_{s^{-}}+\alpha)\alpha^{2}\right]\varphi(X_{s^{-}},\alpha,\xi)\bar{M}(ds,d\alpha,d\xi).
\end{equation*}
Let us admit for the moment:
\bigskip
\begin{lem}
If~\eqref{eq:prop6} holds, then
\begin{equation*}
 \lim_{x\rightarrow
 -\infty}|x|^{p+2}\int_{0}^{\infty}f''(x+\alpha)\alpha^{2}g(x,\alpha)\nu(d\alpha)=p(p+1)V.
\end{equation*}
\end{lem}
This implies that
\begin{equation*}
\lim_{x\rightarrow
-\infty}|x|^{p+2}\int_{0}^{\infty}f''(x+\alpha)\alpha^{2}g(x,\alpha)\nu(d\alpha)<\lim_{x\rightarrow
-\infty}p(2p+1)V(x).
\end{equation*}
Hence, there exists $N\geq K$ such that for all $x\leq -N$, 
\begin{equation*}
\int_{0}^{\infty}f''(x+\alpha)\alpha^{2}g(x,\alpha)\nu(d\alpha)<p(2p+1)\frac{V(x)}{|x|^{p+2}}.
\end{equation*}
Thus, for all $k>0$ satisfying $-kN<X_{0}<-N$,
\begin{equation*}
\begin{split}
f(X_{t\wedge S_{-kN,-N}})\ \leq\ &f(X_{0}) + \int_{0}^{t\wedge
S_{-kN,-N}}\frac{p}{|X_{s}|^{p+1}}\left[\psi(X_{s})+(2p+1)\frac{V(X_{s})}{2|X_{s}|}\right]ds\\
&+ \mathcal{N}_{t\wedge S_{-kN,-N}}.
\end{split}
\end{equation*}
Now if $k\geq 3$, letting $X_{0}=-2N$, it follows from~\eqref{eq:condsup} that
\begin{equation*}
\mathbb{E}(f(X_{t\wedge S_{-kN,-N}}))\leq \frac{1}{(2N)^{p}}.
\end{equation*}
Thus, if we let $t$ tend to $\infty$,
\begin{equation*}
\frac{1}{N^{p}}\mathbb{P}(S_{-kN,-N}=T_{-N}^{\uparrow}))\leq \mathbb{E}\frac{1}{|X_{S_{-kN,-N}}|^{p}}\leq\frac{1}{(2N)^{p}}.
\end{equation*}
Now letting $k$ tend to $\infty$,
\begin{equation*}
\mathbb{P}(T_{-N}^{\uparrow}<\infty)\leq \frac{1}{2^{p}}.
\end{equation*}
Thus, the process is transient, which means
\begin{equation*}
X_{t}\xrightarrow[t\rightarrow \infty]{} -\infty.
\end{equation*}
And since $m=v<\infty$, it follows from Lemma 1 that $\frac{\mathcal{M}_{t}}{t}\rightarrow 0$, hence 
\begin{equation*}
\frac{X_{t}}{t}\xrightarrow[t\rightarrow \infty]{}0.
\end{equation*}
\end{proof}
\begin{proof}[Proof of Lemma 4]
For any $0<\beta<\beta_{0}<1$, if $x<-(1-\beta)^{-1}$,
\begin{equation*}
\begin{split}
|x|^{p+2}\int_{0}^{\infty}f''(x+\alpha)g(x,\alpha)\alpha^{2}\nu(d\alpha)&=|x|^{p+2}\int_{0}^{-\beta
x}f''(x+\alpha)\alpha^{2}g(x,\alpha)\nu(d\alpha)\\
&+|x|^{p+2}\int_{-\beta
x}^{\infty}f''(x+\alpha)\alpha^{2}g(x,\alpha)\nu(d\alpha)\\
&\leq \int_{0}^{-\beta
x}\frac{p(p+1)}{(1-\beta)^{p+2}}\alpha^{2}g(x,\alpha)\nu(d\alpha)\\
&+ |x|^{p+2} p(p+1)\int_{-\beta
x}^{\infty}\alpha^{2}g(x,\alpha)\nu(d\alpha).\\
\end{split}
\end{equation*}
On the other hand,
\begin{equation*}
\begin{split}
 |x|^{p+2}\int_{0}^{\infty}f''(x+\alpha)\alpha^{2}g(x,\alpha)\nu(d\alpha)&\geq
 p(p+1)\int_{0}^{-\beta
 x}\frac{|x|^{p+2}}{|x+\alpha|^{p+2}}\alpha^{2}g(x,\alpha)\nu(d\alpha)\\
 &>p(p+1)\int_{0}^{-\beta x}\alpha^{2}g(x,\alpha)\nu(d\alpha).\\
 \end{split}
\end{equation*}
Letting $x\rightarrow -\infty$ in the two above inequalities, we deduce from~\eqref{eq:prop6}, which holds with $p_{0}$ replaced by $p\leq p_{0}$,
\begin{equation*}
\begin{split}
p(p+1)V&\leq\liminf_{x\rightarrow -\infty}
|x|^{p+2}\int_{0}^{\infty}f''(x+\alpha)\alpha^{2}g(x,\alpha)\nu(d\alpha)\\
&\leq \limsup_{x\rightarrow -\infty}|x|^{p+2}\int_{0}^{\infty}
f''(x+\alpha)\alpha^{2}g(x,\alpha)\nu(d\alpha)\\
&\leq \frac{p(p+1)}{(1-\beta)^{p+2}}V.\\
\end{split}
\end{equation*}
Thus, letting $\beta\rightarrow 0$, it follows that
\begin{equation*}
 |x|^{p+2}\int_{0}^{\infty}f''(x+\alpha)\alpha^{2}g(x,\alpha)\nu(d\alpha)\xrightarrow[x\rightarrow -\infty]{}p(p+1)V.
\end{equation*}
\bigskip
\end{proof}
\begin{rmk}
We have not been able to precise the large time behavior of the process $X_t$ when the
measure $\nu$ is of the type
\begin{equation*}
\nu(d\alpha)\approx \frac{d\alpha}{\alpha^{2+\delta}}\mathbf 1_{\{\alpha\geq 1\}}, \quad 0<\delta\leq \frac{1}{2},
\end{equation*} 
which still satisfies $m<\infty$. In this case, $V=\infty$, $|x\psi(x)|\rightarrow\infty$ as $|x|\rightarrow\infty$, and~\eqref{eq:prop6} also fails.
\end{rmk}

\section{Generalization to the case of a time-variable speed}\label{sec5}
In the following, we treat the case where the speed of environmental
change is a random function of time 
\begin{equation}\label{eq:vt}
v(t)=\int_{0}^{t}v_1(s)ds+\mathcal{R}_t
\end{equation} 
where $v_1$ is a random function such that
\begin{equation*}
\frac{1}{t}\int_{0}^{t}v_1(s)ds\xrightarrow[t\rightarrow\infty]{}\bar v,
\end{equation*}
and $\mathcal R$ is a stochastic process. 
The stochastic equation describing the evolution of phenotypic lag becomes
\begin{equation}\label{eq:processv}
X_{t}= X_{0}-v(t)+\int_{0}^{t}m(X_{s})ds+\mathcal M_{t}.
\end{equation}

As above, we study three cases:

\subsection{The case $\bar{v}>m$}

Here we assume that $\mathcal R$ satisfies the condition
\begin{equation*}
\frac{\mathcal R_t}{t}\xrightarrow[t\rightarrow \infty]{}0.
\end{equation*}
This condition is verified by a Brownian motion for example.
It is easy to see that results~\eqref{eq:prop2}
and~\eqref{eq:martconv} hold in the new context of
equation~\eqref{eq:processv}. Following the steps of the proof in
section 1, we can see that $X_{t}\rightarrow - \infty$ with speed
$\bar{v}-m$.

\subsection{The case $\bar{v}<m$}
Define $\mathcal T$ as the set of bounded stopping times.
Now we assume that there exists $0<c<\infty$ such that 
$\mathbb E\mathcal R_T\leq c$ for all $T\in\mathcal T$.
This condition is verified for example by a process sum of 
a martingale and a bounded process.
In this case, we will prove that the process $X_{t}$ is positive
recurrent. We can see from~\eqref{eq:prop2} and~\eqref{eq:vt}  
that there exist $M,N>0$ such that for $y<-M$ and $t>N$,
\begin{equation*}
m(x)-\frac{1}{t}\int_{0}^{t}v_1(s)ds > \frac{m-\bar{v}}{2}.
\end{equation*} 
We remind that
\begin{equation*}
T_{-M}^{\uparrow}=\inf \{t>0,X_{t}\geq-M\}.
\end{equation*}
For the purpose of notation and without loss of generality, we denote
$T_{-M}^{\uparrow}$ by $T$. Assume that $X_{0}< -M$. It follows that for
all $t>N$,
\begin{equation*}
\mathbb{E}\int_{0}^{t\wedge T}\left[m(X_{s}) -v_1(s)\right]ds<-X_{0}+\mathbb E \mathcal R_{t\wedge T},
\end{equation*}
since $X_s<0$ for $s<T$. We have
\begin{equation*}
\begin{split}
\mathbb{E}\int_{0}^{t\wedge T}\left[m(X_{s})-v_1(s)\right]ds&=\mathbb{E}\mathbf{1}_{T\geq N}\int_{0}^{t\wedge T}\left[m(X_{s})-v_1(s)\right]ds\\
&+\mathbb{E}\mathbf{1}_{T<N}\int_{0}^{t\wedge T}\left[m(X_{s})-v_1(s)\right]ds\\
&<-X_{0}+\mathbb E  \mathcal R_{t\wedge T},\\
\end{split}
\end{equation*}
and hence,
\begin{equation*}
\begin{split}
\frac{m-\bar v}{2}\mathbb{E}\left(\mathbf{1}_{T\geq N}(t\wedge T)\right)&\leq -X_0+\mathbb E  \mathcal R_{t\wedge T}-\mathbb{E}\mathbf{1}_{T<N}\int_{0}^{t\wedge T}\left[m(X_{s})-v_1(s)\right]ds\\
&\leq -X_0 +\mathbb E  \mathcal R_{t\wedge T}+\int_0^N v_1^+(s)ds.
\end{split}
\end{equation*}
Now let $t$ tend to $\infty$, yielding
\begin{equation*}
\mathbb{E}(\mathbf 1_{T\geq N}T)<\frac{2}{m-\bar v}\left(-X_0+\int_0^N v_1^+(s)ds+c \right)<\infty.
\end{equation*}
Thus, $\mathbb{E}(T)<N+\mathbb{E}(\mathbf 1_{T\geq N}T)<\infty$. From here, it is not hard to prove that $\mathbb{E}X_{T^{\uparrow}_{0}}<\infty$. Thus, $X_{t}$ is positive recurrent.

\subsection{The case $\bar{v}=m$}

Here we assume that $\mathcal R_t\equiv 0$.
Even in this case stronger assumptions need to be made. Define
\begin{equation*}
\begin{split}
v_{\sup}&=\sup\limits_{\substack{s}}v_1(s),\\
v_{\inf}&=\inf\limits_{\substack{s}}v_1(s),\\
\psi_{\sup}(x)&=m(x)-v_{\sup},\\
\psi_{\inf}(x)&=m(x)-v_{\inf},\\
\end{split}
\end{equation*}
We define two sets of assumptions:
\bigskip 
\newline
\textbf{\underline{Assumptions A}}
\begin{itemize}
\item $v_{\sup}<\infty$,
\item $\liminf\limits_{x\rightarrow -\infty}|x|\psi_{\sup}(x)>-\frac{V}{2}$.
\end{itemize}
\bigskip
\textbf{\underline{Assumptions B}}
\begin{itemize}
\item $v_{\inf}<\infty$,
\item $\limsup\limits_{x\rightarrow -\infty}|x|\psi_{\inf}(x)<-\frac{V}{2}$.
\end{itemize}
\bigskip
Under the set of assumptions A, we can prove that the process is
recurrent. We have, however, not been able to prove null recurrence in
the case of non-constant $v$.
\begin{proof}[Ideas of Proof]
Apply Lemma 3 to the process in equation~\eqref{eq:processv} with
$f(x)=\log|x|$, with $x<0$. Here $f''$ is decreasing. Hence, as long as
$X_{t}$ remains negative,
\begin{equation*}
\begin{split}
\log|X_{t}|&\leq\log|X_{0}|+\int_{0}^{t}\left(\frac{\psi_{\sup}(X_{s})}{X_{s}}-\frac{V(X_{s})}{2X_{s}^{2}}\right)ds+\int_{0}^{t}\frac{v_{\sup}-v_1(s)}{X_{s}}ds+\mathcal{M'}_{t}\\
&<\log|X_{0}|+\int_{0}^{t}\left(\frac{\psi_{\sup}(X_{s})}{X_{s}}-\frac{V(X_{s})}{2X_{s}^{2}}\right)ds+\mathcal{M'}_{t},\\
\end{split}
\end{equation*}
where $\mathcal{M'}$ is a martingale. 
Then we continue the proof as for the case of constant speed.
\end{proof}
Under the set of assumptions B and hypothesis~\eqref{eq:prop6}, we can
prove that 
\begin{equation*}
X_{t}\xrightarrow[t\rightarrow \infty]{}-\infty\quad\text{ and }\quad\frac{X_{t}}{t}\xrightarrow[t\rightarrow\infty]{}0.
\end{equation*} 
\begin{proof}[Ideas of Proof]
We take the same function $f$ we constructed in the case of constant
speed. We have $f'>0$, and
\begin{equation*}
\begin{split}
f(X_{t})&\leq f(X_{0})+\int_{0}^{t}\psi_{\inf}(X_{s})f'(X_{s})ds+\int_{0}^{t}(v_{\inf}-v_1(s))f'(X_{s})ds\\
&+\frac{1}{2}\int_{0}^{t}\int_{0}^{\infty}f''(X_{s}+\alpha)\alpha^{2}g(X_{s},\alpha)\nu(d\alpha)ds+\mathcal{N'}_{t}\\
&\leq
f(X_{0})+\int_{0}^{t}\psi_{\inf}(X_{s})f'(X_{s})ds+\frac{1}{2}\int_{0}^{t}\int_{0}^{\infty}f''(X_{s}+\alpha)\alpha^{2}g(X_{s},\alpha)\nu(d\alpha)ds+\mathcal{N'}_{t},\\
\end{split}
\end{equation*}
where $\mathcal{N'}$ is a martingale. Then we continue the proof as for
the case of constant speed.
\end{proof}

\paragraph{Acknowledgement} The authors want to thank Michael Kopp for formulating the biological question
which led us to this research, and for many stimulating discussions in the course of our work.

\newpage 
\bibliographystyle{IEEEtranN}
\bibliography{bibarticle} 
\end{document}